\documentclass [11pt,a4paper]{amsart}

\usepackage{amssymb} 
\usepackage{amsfonts} 
\usepackage{amsmath}
\usepackage{amsthm} 
\usepackage{tikz}
\usepackage[utf8]{inputenc}
\usepackage[all]{xy}
\SelectTips{cm}{}

\usepackage[pagebackref,
    ,pdfborder={0 0 0}
    ,urlcolor=black,a4paper,hypertexnames=false]{hyperref}
\hypersetup{pdfauthor={Clara L\"oh},pdftitle=Simplicial volume with Fp coefficients}

\usepackage{caption}

\DeclareMathOperator{\GL}{GL}
\DeclareMathOperator{\vol}{vol}

\DeclareMathOperator{\im}{im}

\newtheorem{lem}{Lemma}[section]
\newtheorem{thm}[lem]{Theorem}
\newtheorem{prop}[lem]{Proposition}
\newtheorem{cor}[lem]{Corollary} 

\theoremstyle{definition}
\newtheorem{defi}[lem]{Definition}

\newtheorem{question}[lem]{Question}

\newtheorem{exa}[lem]{Example}
\newtheorem{rem}[lem]{Remark} 
\newtheorem{ack}[lem]{Acknowledgements}

\newcommand{\N}{\ensuremath {\mathbb{N}}}
\newcommand{\F}{\ensuremath {\mathbb{F}}}
\newcommand{\R} {\ensuremath {\mathbb{R}}}
\newcommand{\Q} {\ensuremath {\mathbb{Q}}}
\newcommand{\Z} {\ensuremath {\mathbb{Z}}}

\DeclareMathOperator{\rk}{rk}

\def\fa#1{\forall_{#1} \quad}
\def\exi#1{\exists_{#1} \quad}

 \makeatletter
 \newcommand\norm{\bBigg@{0.8}}

 \makeatother
 \newcommand{\inparens}[2][flex]{\csname #1l\endcsname(#2%
                                 \csname #1r\endcsname)\mathclose{}}
 \newcommand{\inangles}[2][flex]{\csname #1l\endcsname\langle#2%
                                 \csname #1r\endcsname\rangle\mathclose{}} 
 \newcommand{\innorm}[2][flex]{\csname #1l\endcsname|#2%
                                 \csname #1r\endcsname|\mathclose{}}
 \newcommand{\indnorm}[2][flex]{\csname #1l\endcsname\|#2%
                                 \csname #1r\endcsname\|\mathclose{}}
 \newcommand{\indnorml}[4][flex]{\csname #1l\endcsname\|#2%
                                 \csname #1r\endcsname\|_{#3}^{#4}\mathclose{}}

\newcommand{\sv}[2][flex]{\indnorml[#1]{#2}{\R}{}}%
\newcommand{\isv}[2][norm]{\indnorml[#1]{#2}{\Z}{}}

\newcommand{\pfcl}[2][flex]{\csname #1l\endcsname[#2%
                            \csname #1r\endcsname]}
\newcommand{\ifsv}[2][norm]{\csname #1l\endcsname\bracevert\!#2\!%
                            \csname #1r\endcsname\bracevert}
\newcommand{\stisv}[2][flex]{\indnorml[#1]{#2}{\Z}{\infty}}

\DeclareMathOperator{\supp}{supp}
\newcommand{\svp}[3][norm]{\indnorml[#1]{#2}{(\F_{#3})}{}}
\newcommand{\suv}[3][norm]{\indnorml[#1]{#2}{(#3)}{}}
\newcommand{\stsvp}[3][norm]{\indnorml[#1]{#2}{(\F_{#3})}{\infty}}
\newcommand{\stsuv}[3][norm]{\indnorml[#1]{#2}{(#3)}{\infty}}

\newcommand{\pfc}[1]{%
  \widehat{#1}}

\def\fface#1#2{%
  {#2}\rfloor_{#1}%
}
\def\bface#1#2{%
  {}_{#1}\lfloor{#2}%
}

\def\args{\;\cdot\;}

\newcommand{\actson}{\curvearrowright}

\title{Simplicial volume with $\F_p$-coefficients}

\author{Clara L\"oh}

\subjclass[2010]{57R19, 20E18}

\keywords{simplicial volume, stable integral simplicial volume}

\def\draftinfo{}
\date{\today.\ \copyright{\ C.~L\"oh 2018}. 
    This work was supported by the CRC~1085 \emph{Higher Invariants} 
    (Universit\"at Regensburg, funded by the DFG)\draftinfo}

\begin{document}

\begin{abstract}
  For primes~$p$, we investigate an $\F_p$-version of simplicial
  volume and compare these invariants with their siblings over other
  coefficient rings. We will also consider the associated gradient
  invariants, obtained by stabilisation along finite coverings.
  Throughout, we will discuss the relation between such simplicial
  volumes and Betti numbers.
\end{abstract}
\maketitle

\section{Introduction}

Simplicial volumes measure the size of manifolds in terms of
singular fundamental cycles. Different choices of coefficients
and of counting simplices in singular cycles lead to different
versions of simplicial volume.

The classical simplicial volume~$\sv \args$ of an oriented closed
connected manifold is defined in terms of the $\ell^1$-norm on the
singular chain complex with real coefficients, i.e., we use a weighted
count of singular simplices~\cite{vbc,mapsimvol}
(Section~\ref{subsec:basicdefs} contains a precise definition).

We will study \emph{weightless} simplicial volumes, i.e., we will
ignore the weight of the coefficients and only count the number of
singular simplices. In this way, for \emph{every} coefficient
ring~$R$, we obtain a notion of weightless simplicial volume~$\suv
\args R$ (Definition~\ref{def:suv}). In particular, this leads to an
$\F_p$-version of simplicial volume:

\begin{defi}\label{def:svp}
  Let $M$ be an oriented closed connected $n$-manifold and let $p \in \N$
  be a prime number. Then the \emph{$\F_p$-simplicial volume~$\svp M p$ of~$M$} is
  defined as 
  \begin{align*}
    \min \biggl\{ m \in \N \biggm|
    \sum_{j=1}^m a_j \cdot \sigma_j \in C_n(M;\F_p)
    \text{ is an $\F_p$-fundamental cycle of~$M$} \biggr\} 
    .
  \end{align*}
\end{defi}

The following natural questions emerge:
\begin{enumerate}
  \def\theenumi{(\Alph{enumi})}
  \def\labelenumi{\theenumi}
\item\label{q:topmeaning} What is the topological/geometric meaning of these invariants?
\item\label{q:topinv} How do these invariants relate to other topological invariants?
\item\label{q:diffp} (How) Does the choice of the prime number affect simplicial volume?
\item\label{q:Q} Do these invariants relate to the weightless $\Q$-simplicial volume?
\item\label{q:stable} What happens if $\F_p$-simplicial volumes are stabilised along a tower of
  finite coverings?
\end{enumerate}
In the present paper, we will study these questions, keeping a focus on
the comparison with Betti numbers. In order to simplify the discussion,
we will start with Questions~\ref{q:diffp} and \ref{q:Q}. 

\subsection{Comparison of coefficients}

Using different finite fields as coefficients for singular homology,
in general, leads to different Betti numbers. But the universal coefficient
theorem implies that this dependence on the coefficients is rather limited:
If $X$ is a finite CW-complex and $k \in \N$, then for all but finitely many
primes~$p \in \N$ we have
\[ b_k(X;\F_p) = b_k(X;\Q).
\]
The situation for weightless simplicial volumes is similar: In
general, different prime numbers will lead to different simplicial
volumes over the corresponding finite fields
(Example~\ref{exa:diffp}). But for every manifold, this exceptional
behaviour is limited to a finite number of primes:

\begin{thm}\label{thm:fpq}
  Let $M$ be an oriented closed connected manifold. Then for all but finitely
  many primes~$p \in \N$ we have
  \[ \svp M p = \suv M \Q.
  \]
\end{thm}

In the same spirit, as for Betti numbers, the weightless simplicial
volumes of fields with equal characteristic coincide
(Theorem~\ref{thm:equalchar}).

In contrast with the Betti number case, the weightless $\Q$-simplicial
volume does \emph{not} necessarily coincide with the weightless
$\Z$-simplicial volume (Proposition~\ref{prop:minimal} shows that $\R
P^3$ is an example of this type) and that the ordinary integral
simplicial volume~$\isv \args$ (Section~\ref{subsec:basicdefs}) tends
to be much bigger than the weightless $\Q$-simplicial volume:

\begin{thm}\label{thm:gap}
  For every~$K \in \N$ there exists an oriented closed connected $3$-manifold with
  \[ \isv M > K \cdot \suv M \Q.  
  \]
\end{thm}

\begin{question}[weight problem]\label{q:weight}
  Let $M$ be an oriented closed connected manifold.  What is the
  difference between integral simplicial volume~$\isv M$ and the
  weightless integral simplicial volume~$\suv M
  \Z$\;?
\end{question}

\subsection{Topological meaning}

We will now come to questions~\ref{q:topmeaning} and~\ref{q:topinv}.
Weightless simplicial volumes give Betti number bounds and provide
obstructions against domination of manifolds
(Definition~\ref{def:dom}). We explain this in more detail:

If $M$ is an oriented closed connected manifold and $p \in \N$ is
a prime number, then Poincar\'e duality shows that for all~$k \in \N$
we have
\[ b_k(M;\F_p) \leq \svp M p 
\]
(Proposition~\ref{prop:PD}). This estimate can be used to calculate
weightless simplicial volumes in simple cases
(Section~\ref{subsec:hombounds}); conversely, estimates of this type
lead to homology gradient bounds, when stabilising along towers of
finite coverings (Section~\ref{subsec:istable}).

Classical simplicial volume is a homotopy invariant and compatible with mapping
degrees: If $f \colon M \longrightarrow N$ is a continuous map between oriented
closed connected manifolds of the same dimension, then
\[ \mathopen| \deg f| \cdot \sv N \leq \sv M. 
\]
Therefore, classical simplicial volume gives a priori bounds on the set
of possible mapping degrees~\cite{vbc}.

Similarly, also weightless simplicial volumes are homotopy invariants
and they satisfy the following estimates
(Proposition~\ref{prop:degreemon}): Let $f \colon M \longrightarrow N$
be a continuous map between oriented closed connected manifolds of the
same dimension.
\begin{itemize}
\item If $p \in \N$ is a prime number with~$p \nmid \deg f$, then $\svp N p \leq \svp M p$.
\item If $\deg f \neq 0$, then $\suv N \Q \leq \suv M\Q$.
\end{itemize}
In addition to the similarity with the degree estimate for classical
simplicial volume these monotonicity properties are similar to the
corresponding monotonicity statements for Betti numbers with
$\F_p$-coefficients and $\Q$-coefficients, respectively. In
particular, the weightless $\Q$-simplicial volume complements the
Betti number obstruction against domination of manifolds
(Section~\ref{subsec:domination}).

\subsection{Stabilisation along finite coverings}\label{subsec:istable}

Finally, we address the last question~\ref{q:stable} on stabilisation
of $\F_p$-simplicial volumes along towers of finite coverings.

\begin{defi}\label{def:stsvp}
  Let $M$ be an oriented closed connected manifold and let $p \in \N$
  be a prime number. Then the \emph{stable $\F_p$-simplicial volume of~$M$}
  is defined as
  \[ \stsvp M p := \inf
  \Bigl\{ \frac{\svp N p}d 
  \Bigm| \text{$d \in \N$, $N \rightarrow M$ is a $d$-sheeted covering}
  \Bigr\}
  \in \R_{\geq 0}.
  \]
\end{defi}

\begin{rem}[inherited vanishing]
  Let $M$ be an oriented closed connected manifold and let $p \in \N$ be
  a prime number. Then reduction modulo~$p$ shows that
  \[ \stsvp M p \leq \stisv M.
  \]
  In particular, all known vanishing results for the stable integral
  simplicial volume~$\stisv M$ imply corresponding vanishing results
  for~$\stsvp M p$. This includes, for example, the case of aspherical
  manifolds with residually finite amenable fundamental
  group~\cite{FLPS}, smooth aspherical manifolds with non-trivial
  $S^1$-action~\cite{fauser}, and graph
  manifolds~\cite{ffl}. Moreover, also aspherical manifolds with
  small enough amenable covers have vanishing stable weightless
  simplicial volume~\cite{sauer} 
  (Example~\ref{exa:smallamenable}). 
\end{rem}
  
In addition, we have the following vanishing phenomenon in the presence of
non-trivial self-maps: If $M$ is an oriented closed connected manifold
that admits a self-map of non-trivial degree, then the classical
simplicial volume satisfies~$\sv M = 0$ (this follows from the degree
estimate). If $M$ is aspherical and the fundamental group is
residually finite, then this vanishing also carries over to
$L^2$-Betti numbers~\cite[Theorem~14.40]{lueck}. In the same way, we obtain
vanishing of stable weightless simplicial volumes:

\begin{thm}\label{thm:selfmap}
  Let $M$ be an oriented closed connected aspherical manifold with residually finite
  fundamental group. If $M$ admits a continuous map~$f \colon M \longrightarrow M$
  with $\deg f \not \in \{-1,0,1\}$, then for all primes~$p \in \N$ with $p \nmid \deg f$
  we have
  \[ \stsvp M p = 0 = \stsuv M \Q.
  \]
\end{thm}

It should be noted that it is an open problem to determine whether the
same conclusion also holds for~$\stisv M$ instead of~$\stsvp M p$.

As in the case of stable integral simplicial volume~\cite[Theorem~6.6,
  Remark~6.7]{loehpagliantini}, stable $\F_p$-simplicial volumes admit
a description in terms of the probability measure preserving action
of~$\pi_1(M)$ on its profinite completion~$\pfc{\pi_1(M)}$ (see
Section~\ref{subsec:dyndef} for the definitions):

\begin{thm}[dynamical view]\label{thm:dynview}
  Let $M$ be an oriented closed connected manifold with residually finite fundamental
  group and let $p \in \N$ be a prime number. Then
  \[ \stsvp M p = \suv M {\pfc{\pi_1(M)};\F_p}.
  \]
\end{thm}

We conclude with a brief outlook on the relation with homology
gradients:

\begin{rem}[homology gradients]
  Let $M$ be an oriented closed connected manifold and let $p \in \N$
  be a prime number. Let $(\Gamma_j)_{j \in \N}$ be a descending chain
  of finite index subgroups of~$\pi_1(M)$ and let $(M_j)_{j \in \N}$
  be the corresponding tower of covering manifolds. Then the Betti
  number bound of Proposition~\ref{prop:PD} yields the corresponding
  homology gradient bound
  \[ \limsup_{j \rightarrow \infty} \frac{b_k(M_j;\F_p)}{[\pi_1(M):\Gamma_j]}
  \leq \lim_{j \rightarrow \infty} \frac{\svp {M_j} p}{[\pi_1(M):\Gamma_j]}
  = \inf_{j \in \N} \frac{\svp {M_j} p}{[\pi_1(M):\Gamma_j]}.
  \]
\end{rem}

Similar to the case of homology gradients, we are therefore led to
the following open problems:

\begin{question}[approximation problem]
  Let $M$ be an oriented closed connected manifold with residually
  finite fundamental group.  For which (if any) primes~$p \in \N$ do
  we have
  \[ \stsvp M p = \stsuv M \Q\; ?
  \]
\end{question}

\begin{question}[chain problem]
  Let $M$ be an oriented closed connected manifold with residually
  finite fundamental group, let $p \in \N$ be a prime, let
  $(\Gamma_j)_{j \in \N}$ be a descending chain of finite index
  subgroups of~$\pi_1(M)$, and let $(M_j)_{j \in \N}$ be the associated
  sequence of covering manifolds of~$M$.  How does the value
  \[ \lim_{j \rightarrow \infty} \frac{\svp {M_j} p}{[\pi_1(M):\Gamma_j]}
  \]
  depend on the choice of this chain~$(\Gamma_j)_{j \in \N}$\;?
\end{question}

\begin{question}[Euler characteristic problem]
  Gromov asked the following question~\cite[p.~232]{gromovggt}: Let $M$ be
  an oriented closed connected aspherical manifold~$M$ with~$\sv M = 0$.
  Does this imply~$\chi(M) = 0$\;?

  The Betti number estimate shows that 
  $\bigl|\chi(M)\bigr| \leq (\dim M+1) \cdot \stsvp M p$ for all primes~$p$
  (Proposition~\ref{prop:chi}).

  Hence, we arrive at the following question: Let $M$ be an oriented
  closed connected aspherical manifold with~$\sv M = 0$ (and
  residually finite fundamental group). Does there exist a prime
  number~$p \in \N$ with~$\stsvp M p = 0$\;?
\end{question}

\subsection*{Organisation of this article}

We introduce weightless simplicial volumes in
Section~\ref{sec:wlsimvol} and establish some basic properties,
including a proof of Theorem~\ref{thm:gap}. In
Section~\ref{sec:comparison} we prove the comparison theorem
Theorem~\ref{thm:fpq}. Stable weightless simplicial volumes are
studied in Section~\ref{sec:stable}, where we prove Theorem~\ref{thm:selfmap}
and Theorem~\ref{thm:dynview}.

\section{Weightless simplicial volumes}\label{sec:wlsimvol}

We will introduce weightless simplicial volumes and the special case
of $\F_p$-simplicial volumes. Moreover, we will collect some basic
properties of these invariants.

\subsection{Basic definitions}\label{subsec:basicdefs}

If $R$ is a subring of~$\R$, then the \emph{$R$-simplicial volume} of an oriented
closed connected $n$-manifold~$M$ is defined as
\begin{align*}
  \| M \|_R := \inf \biggl\{ \sum_{j=1}^m |a_j| \biggm| \;
  & \sum_{j=1}^m a_j \cdot \sigma_j \in C_n(M;R)
  \\
  & \text{is an $R$-fundamental cycle of~$M$}\biggr\} \in \R_{\geq 0}.
\end{align*}
Here, $|\cdot|$ denotes the standard absolute value on~$\R$ and $C_n(M;R)$
is the $n$-th chain group of the singular chain complex of~$M$. Classical
simplicial volume, as introduced by Gromov~\cite{vbc}, is $\sv M$.

If $R$ is a commutative ring with unit, in general, we will not
find a norm on~$R$ in the usual sense; but we can always use the trivial
size function
\begin{align*}
  | \cdot | \colon R & \longrightarrow \R_{\geq 0} \\
  x & \longmapsto \begin{cases} 0 & \text{if~$x = 0$}\\
    1 & \text{if~$x \neq 0$}
  \end{cases}
\end{align*}
on the coefficients, 
which leads to the counting ``norm'' on the singular chain complex (thereby forgetting
the weights of the individual simplices):
If $X$ is a topological space and $k\in \N$, then we define 
\begin{align*}
  \biggl| \sum_{j=1}^m a_j \cdot \sigma_j\biggr| := \sum_{j=1}^m |a_j| = m
\end{align*}
for all singular chains~$\sum_{j=1}^m a_j \cdot \sigma_j \in C_k(X;R)$ in reduced form. 
This ``norm'' on the singular chain complex results in the following definition:

\begin{defi}[weightless simplicial volume]\label{def:suv}
  Let $R$ be a commutative ring with unit and let $M$ be an oriented closed connected
  manifold of dimension~$n$.
  Then the \emph{weightless $R$-simplicial volume} of~$M$ is defined as
  \begin{align*}
    \suv M R := \min \biggl\{ m \in \N \biggm|\;
     & \sum_{j=1}^m a_j \cdot \sigma_j \in C_n(M;R)
    \\
     & \text{is an $R$-fundamental cycle of~$M$} \biggr\} 
    \in \N.
  \end{align*}
\end{defi}

If $p \in \N$ is a prime number, then for the ring~$\F_p$ we obtain Definition~\ref{def:svp}.

Of course, we can extend this definition of~$\suv {\args} R$ to all singular
homology classes with $R$-coefficients
by minimising the counting norm over all singular cycles representing the
given homology class.

\begin{rem}\label{rem:trivest}
  Let $R$ be a (non-trivial) commutative ring with unit and let $M$ be
  an oriented closed connected manifold.
  \begin{enumerate}
  \item If $M$ is non-empty, then $\suv M R \neq 0$, whence $\suv M R \geq 1$.
  \item Because every $\Z$-fundamental cycle gives rise to an $R$-fundamental cycle,
    we have
    \[ \suv M R \leq \suv M \Z \leq \isv M.
    \]
  \end{enumerate}
\end{rem}

\subsection{Degrees and domination}\label{subsec:domination}

Weightless simplicial volumes can serve as an obstruction against domination of manifolds.

\begin{prop}[degree monotonicity]\label{prop:degreemon}
  Let $M$ and $N$ be oriented closed connected manifolds of the same dimension and  
  let $f \colon M \longrightarrow N$ be a continuous map.
  \begin{enumerate}
  \item If $R$ is a commutative ring with unit and $\deg f$ is a unit in~$R$, then
    \[ \suv N R \leq \suv M R.
    \]
  \item In particular: If $\deg f \neq 0$, then $\suv N \Q \leq \suv M \Q$.
    If $p \in \N$ is prime and $p \nmid \deg f$, then $\svp N p \leq \svp M p$.
  \end{enumerate}
\end{prop}
\begin{proof}
  Let $c = \sum_{j=1}^k a_j \cdot \sigma_j \in C_n(M;R)$ be an $R$-fundamental
  cycle that is in reduced form. Because $\deg f$ is a unit in~$R$, the chain
  \[ c' := \sum_{j=1}^k \frac 1 {\deg f} \cdot a_j \cdot f \circ \sigma_j \in C_n(N;R)
  \]
  is an $R$-fundamental cycle of~$N$. In particular,
  \[ \suv N R \leq |c'| \leq k = |c|.
  \]
  Taking the minimum over all $R$-fundamental cycles~$c$ of~$M$ shows
  that $\suv N R \leq \suv M R$.
  
  The second part is a special case of the first part.
\end{proof}

\begin{defi}[domination]\label{def:dom}
  Let $M$ and $N$ be oriented closed connected manifolds of the same dimension.
  The manifold~$M$ \emph{dominates}~$N$ if there exists a map~$M \longrightarrow N$
  of non-zero degree.
\end{defi}

\begin{cor}
  Let $M$ and $N$ be oriented closed connected manifolds of the same dimension
  and let $F$ be a field of characteristic~$0$. If $M$ dominates~$N$, then
  \[ \suv M F \geq \suv N F.
  \]
\end{cor}
\begin{proof}
  This is merely a reformulation of Proposition~\ref{prop:degreemon}, using
  the domination terminology.
\end{proof}

\subsection{Homology bounds and some examples}\label{subsec:hombounds}

We will now explain how homology bounds in terms of weightless
simplicial volumes can be used to compute first examples. 

\begin{prop}[homology bounds]\label{prop:PD}
  Let $M$ be an oriented closed connected manifold and let $R$ be a
  commutative ring with unit.
  \begin{enumerate}
  \item For all~$\alpha \in H_*(M;R)$ we have~$\suv \alpha R \leq \suv M R$.
  \item
    If $R$ is a principal ideal domain and $k \in \N$, then 
    \[ \rk_R H_k(M;R) \leq \suv M R.
    \]
  \end{enumerate}
\end{prop}
\begin{proof}
  Like all results of this type, this is based on exploiting the
  explicit description of the Poincar\'e duality map.

  Let $n := \dim M$,
  let $c = \sum_{j=1}^m a_j \cdot \sigma_j \in C_n(M;R)$ be an $R$-fundamental cycle of~$M$
  with $m = \suv M R$,
  and let $k \in \N$. Then the Poincar\'e duality map
  \begin{align*}
    \args \cap [M]_R \colon H^{n-k}(M;R) & \longrightarrow H_k(M;R)
    \\
      [f] & \longmapsto
      (-1)^{(n-k) \cdot k} \cdot 
      \biggl[ \sum_{j=1}^m a_j \cdot f(\bface {n-k} {\sigma_j}) \cdot \fface k {\sigma_j} 
      \biggr]
  \end{align*}
  is an isomorphism of $R$-modules. Clearly, the elements on the right hand side
  have weightless norm at most~$m = |c| = \suv M R$. This proves the first part.

  Moreover, the $R$-module~$H_k(M;R)$ is a quotient of a submodule of
  an $R$-module that is generated by $m$~elements. If $R$ is a
  principal ideal domain, this implies~$\rk_R H_k(M;R) \leq m$.
\end{proof}

\begin{exa}[odd-dimensional projective spaces]\label{exa:diffp}
  Let $n \in \N$ be odd. 
  If $p \in \N$ is an odd prime, then
  \[ \svp {\R P^n} p
     = 1
  \]
  because $\R P^n$ is dominated by~$S^n$ through a map of degree~$2$ and $2$
  is invertible in~$\F_p$ (Proposition~\ref{prop:degreemon}).

  In contrast, 
  \[ \svp {\R P^n} 2 = 2.
  \]
  Indeed, on the one hand, $\svp {\R P^n} 2 \leq \isv {\R P^n} =
  2$~\cite[Proposition~4.4]{loehsmallisv} (the cited argument
  also generalises to odd dimensions bigger than~$3$). On the other hand, let us
  assume for a contradiction that $\svp {\R P^n} 2 = 1$. Then the
  Poincar\'e duality argument from Proposition~\ref{prop:PD} shows
  that the non-trivial class in~$H_2(\R P^n;\F_2) \cong \F_2$ can be represented
  by a single singular $2$-simplex; however, a single singular
  $2$-simplex cannot be an $\F_2$-cycle (for parity reasons). This
  contradiction shows that~$\svp {\R P^n} 2 = 2$.
\end{exa}

\begin{prop}[minimal weightless simplicial volume]\label{prop:minimal}
  Let $M$ be an oriented closed connected manifold and let $n := \dim M$.
  \begin{enumerate}
  \item Then $\suv M \Q = 1$ if and only if $n$ is odd and $M$ is dominated by~$S^n$.
  \item Then $\suv M \Z = 1$ if and only if $n$ is odd and $M \simeq S^n$.
  \end{enumerate}
\end{prop}

\begin{proof}
  If $n$ is odd, then $\isv {S^n} = 1$ (we can wrap~$\Delta^n$ around~$S^n$
  with constant face maps). Hence: If $M \simeq S^n$, then $\suv M \Z = 1$.
  If $M$ is dominated by~$S^n$, then $1 \leq \suv M \Q \leq \suv {S^n} \Q =1$.

  Conversely, if $\suv M\Q =1$, then there is a singular
  simplex~$\sigma \colon \Delta^n \longrightarrow M$ that is a cycle
  that represents a non-zero homology class in~$H_n(M;\Q)$ and
  in~$H_n(M;\Z)$.  In particular, there is an~$m \in \Z \setminus
  \{0\}$ with~$\| m \cdot [M]_\Z\|_{\Z} = 1$. It is known that this
  implies that $n$ is odd and that $M$ is dominated
  by~$S^n$~\cite[Theorem~3.2]{loehsmallisv}.

  Similarly, if $\suv M\Z = 1$, then there is a singular
  simplex~$\sigma \colon \Delta^n \longrightarrow M$ and an~$m \in \Z
  \setminus \{0\}$ such that $m \cdot \sigma$ is a fundamental cycle
  of~$M$. Because $[m \cdot \sigma] = [M]_\Z$ is a generator
  of~$H_n(M;\Z) \cong \Z$, we obtain~$m \in \{-1,1\}$. Therefore, $\isv M = 1$,
  which implies that $n$ is odd and $M \simeq S^n$~\cite[Theorem~1.1]{loehsmallisv}.
\end{proof}

We will now establish Theorem~\ref{thm:gap}, by proving the following generalisation:

\begin{thm}\label{thm:gapgen}
  Let $F$ be a field. 
  For every~$K \in \N$ there exists an oriented closed connected $3$-manifold with
  \[ \isv M > K \cdot \suv M F.  
  \]
\end{thm}

\begin{proof}
  Let $p$ be the characteristic of~$F$ and let $N := \N \setminus p \cdot \N$. 
  For~$n \in \N$, we write~$L(n,1)$ for the associated $3$-di\-men\-sion\-al lens space.
  Then $L(n,1)$ is covered, whence dominated, by~$S^3$ through a map of degree~$n$;
  therefore, for all~$n \in N$ we obtain
  \[ \suv[big] {L(n,1)} F \leq \suv {S^3} F \leq \isv {S^3} = 1
  \]
  (by Proposition~\ref{prop:degreemon}).
  
  On the other hand, the manifolds in the
  sequence~$\bigl(L(n,1)\bigr)_{n \in N}$ are pairwise
  non-homeomorphic. Therefore, the sequence~$\bigl( \isv
  {L(n,1)}\bigr)_{n \in N}$ is
  unbounded by the finiteness result for the integral simplicial volume~$\isv \args$ on
  $3$-manifolds~\cite[Proposition~5.3]{loehsmallisv}.
\end{proof}

Classical simplicial volume (with real coefficients) is known to behave
well with respect to products~\cite{vbc}. For integral simplicial
volume of a product manifold, there is an upper bound in terms of the
products of the integral simplicial volumes of the factors; however,
it is unknown whether the integral simplicial volume also satisfies a
corresponding estimate from below. 

\begin{prop}[product estimate]
  Let $M$ and $N$ be oriented closed connected manifolds and let $R$
  be a commutative ring with unit.  Then
  \[ \max \bigl(\suv M R, \suv N R \bigr)
  \leq \suv {M \times N} R
  \leq {\dim M + \dim N \choose \dim M} \cdot \suv M R \cdot \suv N R.
  \]
\end{prop}
\begin{proof}
  The upper estimate can be shown via the explicit description of the
  homological cross product on the level of singular chains through the
  shuffle product; this is similar to the case of ordinary simplicial
  volume~\cite[Theorem~F.2.5]{benedettipetronio}.

  For the lower estimate, we argue as follows: Let $y \in N$ be a
  point and let $i \colon M \longrightarrow M \times \{y\}
  \longrightarrow M \times N$ and $p \colon M \times N \longrightarrow
  M$ be the corresponding inclusion and projection, respectively; moreover,
  we consider
  \[ \alpha := H_{\dim M}(i ;R) [M]_R \in H_{\dim M}(M \times N;R).
  \]
  Using Proposition~\ref{prop:PD}, we obtain
  \begin{align*}
    \suv M R
    = \suv[big] {H_{\dim M}(p;R) (\alpha)} R
    \leq \suv \alpha R
    \leq \suv {M \times N} R.
  \end{align*}
  In the same way, we also have~$\suv N R \leq \suv {M \times N} R$.
\end{proof}

Using the Betti number estimate, we can also give an answer to the
weight problem (Question~\ref{q:weight}) in a very simple case: 

\begin{prop}\label{prop:bettieq}
  Let $M$ be an oriented closed connected manifold and let $k \in \N$.
  Then $b_k(M;\Z) = \isv M$ if and only if $b_k(M;\Z) = \suv M \Z$. 
\end{prop}
\begin{proof}
  The Betti number estimate (Proposition~\ref{prop:PD}) and the universal
  coefficient theorem imply that  
  \[ b_k(M;\Z) \leq b_k(M;\F_p) \leq \svp M p \leq \suv M \Z \leq \isv M. 
  \]
  for all primes~$p \in \N$. Hence, if $b_k(M;\Z) = \isv M$, then
  $b_k(M;\Z) = \suv M \Z$ (and also $b_k(M;\Z) = \svp M p$ for all primes~$p$).

  Conversely, let us suppose that $b_k(M;\Z) = \suv M \Z$. Let $n := \dim M$ and let 
  $c = \sum_{j=1}^m a_j \cdot \sigma_j \in C_n(M;\Z)$ be a $\Z$-fundamental
  cycle of~$M$ that satisfies~$m = \suv M \Z$. \emph{Assume} for a contradiction
  that there exists a~$j \in \{1,\dots,m\}$ with~$a_j \not\in \{-1,1\}$. Then
  there is a prime number~$p \in \N$ that divides~$a_j$. Hence, the term~$a_j \cdot \sigma_j$
  vanishes after reduction modulo~$p$ and so
  \[ b_k(M ;\Z) \leq \svp M p \leq m - 1 < m = \suv M \Z = b_k(M;\Z), 
  \]
  which is a contradiction. This shows that $a_1,\dots, a_m \in \{-1,1\}$.
  Therefore, we have
  \[ \isv M \leq |c|_1 = m = \suv M \Z \leq \isv M,
  \]
  and thus~$\isv M = \suv M \Z$.
\end{proof}

\begin{exa}[surfaces]\label{exa:surface}
  Let $\Sigma$ be an oriented closed connected (non-empty) surface
  and let $p \in \N$ be a prime.

  If $\Sigma \cong S^2$, then $\svp \Sigma p \geq 2$ (for parity reasons)
  and $\svp \Sigma p \leq 2$ (straightforward construction). Hence,
  $\svp {S^2}p = 2$. Similarly, $\suv {S^2} R =2$ for every non-trivial
  commutative ring~$R$ with unit.

  We will now consider the case~$\Sigma \not\simeq S^2$. Let $g \in \N$
  denote the genus of~$\Sigma$. Because $\isv \Sigma = 4 \cdot g -2$~\cite[Proposition~4.3]{loehsmallisv},
  we obtain the upper estimate
  \[ \svp \Sigma p \leq \suv \Sigma \Z \leq \isv \Sigma = 4 \cdot g -2.
  \]
  On the other hand, the Betti number estimate (Proposition~\ref{prop:PD})
  yields
  \[ 2 \cdot g = b_1(\Sigma;\F_p) \leq \svp \Sigma p
  \quad\text{and}\quad
     2 \cdot g = b_1(\Sigma;\Z) \leq \suv \Sigma\Z.
  \]
  In particular, $\svp {S^1 \times S^1} p = 2 = \suv {S^1 \times S^1} \Z$.
  However, if $g\geq 2$, the exact values of~$\svp \Sigma p$ or $\suv \Sigma
  \Z$ are not known.

  For the weightless simplicial volume with $\Z$-coefficients, we can
  at least improve the bound $2 \cdot g \leq \suv \Sigma \Z$ to a
  strict inequality:
  \emph{Assume} for a contradiction that $2 \cdot g = \suv \Sigma \Z$.
  Then Proposition~\ref{prop:bettieq} implies that
  $2 \cdot g = \isv \Sigma = 4 \cdot g -2$, which is impossible for~$g \geq 2$.
  Hence, we obtain $2 \cdot g < \suv \Sigma \Z$. 
\end{exa}

\section{Comparison theorems}\label{sec:comparison}

We will now focus on the comparison theorem (Theorem~\ref{thm:fpq})
that relates weightless simplicial volumes over~$\F_p$ and~$\Q$.  In
order to promote fundamental cycles over~$\F_p$ to fundamental cycles
over~$\Q$ (while keeping control on the number of simplices), we will
consider the combinatorial types of simplices and encode the relevant
information in (finite) systems of linear equations.

\subsection{Combinatorics of singular cycles}

We record the combinatorial structure of singular chains in
certain generalised simplicial complexes. These complexes
are specified by a set of $n$-simplices and an adjacency relation
between the faces of these $n$-simplices. 

\begin{defi}[model complexes]
  Let $n \in \N$. An \emph{$n$-dimensional model complex} is a pair~$Z
  = (S,\sim)$ consisting of a set~$S$ (the set of \emph{simplices
    of~$Z$}) and an equivalence relation~$\sim$ on~$S \times \{0,\dots,n\}$.

  Two $n$-dimensional model complexes~$(S,\sim)$ and $(S', \sim')$ are
  \emph{isomorphic} if there exists a bijection~$f \colon S
  \longrightarrow S'$ between their sets of simplices that is
  compatible with the adjacency relations, i.e.,
  \[ \fa{(s,j),(t,k) \in S \times \{0,\dots,n\}} (s,j) \sim (t,k) \Longleftrightarrow \bigl(f(s),j\bigr) \sim' \bigl(f(t),k\bigr).
  \]
\end{defi}

\begin{defi}[model]
  Let $M$ be a topological space, let $n \in \N$, let $R$ be a
  commutative ring with unit, and let $c = \sum_{j=1}^m a_j \cdot
  \sigma_j \in C_n(M;R)$ be a singular chain in reduced form. The \emph{model
    of~$c$} is the $n$-dimensional model complex~$Z = (\{\sigma_1,
  \dots, \sigma_m\},\sim)$, where $\sim$ is given by
  \[ \fa{\sigma,\tau \in \{\sigma_1, \dots, \sigma_m\}}
  \fa{j,k \in \{0,\dots,n\}}
    (\sigma,j) \sim (\tau,k) \Longleftrightarrow
    \sigma \circ \partial_j = \tau \circ \partial_k.
  \]
  Here $\partial_j \colon \Delta^{n-1} \longrightarrow \Delta^n$ denotes
  the inclusion of the $j$-th face of~$\Delta^n$.
\end{defi}

\subsection{Translation to linear algebra}

Because the model of a singular chain stores which faces are equal,
we can encode the property of being a cycle into a linear equation.
(The linear equation is redundant, but it does have the advantage
that the description is simple.)

\begin{defi}[cycle matrix]
  Let $n \in \N$ and let $Z = (S,\sim)$ be an $n$-dimensional model
  complex. If $(s,i), (t,j) \in S$, then we write
  \[ r_{(s,i), (t,j)} := \begin{cases}
    0 & \text{if $(s,i) \not \sim (t,j)$} \\
    (-1)^{j} & \text{if $(s,i) \sim (t,j)$}.
    \end{cases}
  \]
  The \emph{cycle matrix of~$Z$} is the matrix~$A =
  (a_{(s,i),t})_{((s,i),t) \in (S \times \{0,\dots,n\})\times S}$ given by
  \[ a_{(s,i),t} :=
  \sum_{j=0}^n r_{(s,i), (t,j)}
  \]
  for all $(s,i) \in S \times \{0,\dots,n\}$, $t \in S$.
\end{defi}

\begin{lem}\label{lem:cycleencoding}
  Let $M$ be a topological space, let $n \in \N$, let $R$ be a
  commutative ring with unit. Let $c = \sum_{j=1}^m a_j \cdot
  \sigma_j \in C_n(M;R)$ be a singular chain in reduced form,
  and let $A$ be the cycle matrix of the model of~$c$. Then
  $c$ is a cycle if and only if
  \[ A \cdot
  \begin{pmatrix}
    a_1
    \\
    \vdots
    \\
    a_m
  \end{pmatrix}
  = 0.
  \]
  Here, we view $A$ as a matrix over~$R$ via the canonical unital ring
  homomorphism~$\Z \longrightarrow R$.
\end{lem}
\begin{proof}
  Let $(s,i) \in S \times \{0,\dots,n\}$ and let $A_{(s,i)}$ be the
  $(s,i)$-row of~$A$. By construction of the cycle matrix~$A$, the value
  \[ A_{(s,i)} \cdot 
  \begin{pmatrix}
    a_1
    \\
    \vdots
    \\
    a_m
  \end{pmatrix}
  \quad \text{in~$R$}
  \]
  is the total contribution of the singular simplex~$s
  \circ \partial_i$ in the chain~$\partial c = \sum_{j=1}^m \sum_{\ell =
    0}^n (-1)^\ell \cdot a_j \cdot \sigma_j \circ \partial_\ell \in
  C_{n-1}(M;R)$.
  
  Hence, $A \cdot (a_1,\dots, a_m)^\top = 0$ if and only if $\partial c = 0$.
\end{proof}

\subsection{Basics on solution spaces}

For the sake of completeness, we recall two facts on solutions of
linear equations:

\begin{defi}
  Let $k,m \in \N$, let $A \in M_{k \times m}(\Z)$ and $b \in
  \Z^k$. If $R$ is a commutative ring with unit, then we view~$A$ and
  $b$ as a matrix/vector over~$R$ and we write
  \begin{align*}
    L_A(R)
    & := \{ x \in R^m \mid A \cdot x = 0 \} \subset R^m
    \\
    L_{A,b}(R)
    & := \{ x \in R^m \mid A \cdot x = b \} \subset R^m
  \end{align*}
  for the corresponding sets of solutions over~$R$.
\end{defi}

\begin{lem}\label{lem:equalcharsolv}
  Let $F$ and $E$ be fields of the same characteristic.  
  Let $k, m \in \N$ and let $A \in M_{k \times m}(\Z)$, $b \in \Z^k$.
  Then
  \[ L_{A,b}(F) \neq \emptyset \Longleftrightarrow L_{A,b}(E) \neq \emptyset.
  \]
\end{lem}
\begin{proof}
  By transitivity, it suffices to consider the case where $F$ is the
  prime field of~$E$.  In view of Gaussian elimination, there exist
  invertible matrices~$S \in \GL_k(F)$ and $T \in \GL_m(F)$ such that
  \[ \widetilde A := S \cdot A \cdot T \in M_{k \times m}(F)
  \]
  is in row echelon form (because the original equation is defined over~$\Z$,
  whence over~$F$, we only need matrices~$S$ and $T$ with entries in~$F$). 

  Let $\widetilde b := S \cdot b \cdot T \in F^k$. Then
  \[ L_{A,b}(F) \neq \emptyset
     \Longleftrightarrow 
     L_{\widetilde A, \widetilde b}(F) \neq \emptyset
  \]
  and similarly for~$E$. Moreover, because $\widetilde A$ is in row echelon
  form and $F$ is a subfield of~$E$, we have
  \[ L_{\widetilde A, \widetilde b}(F) \neq \emptyset \Longleftrightarrow
     L_{\widetilde A, \widetilde b}(E) \neq \emptyset.
  \]
  Combining these equivalences proves the claim.
\end{proof}

\begin{lem}\label{lem:reduction}
  Let $k,m \in \N$ and let $A \in M_{k \times m}(\Z)$.
  Suppose that $P \subset \N$ is an infinite set of primes with
  \[ \fa{p \in P} L_A(\F_p) \neq \{0\}.
  \]
  Then $L_A(\Z) \neq \{0\}$ and for all but finitely many~$p \in P$ we have
  (where $\pi_p \colon \Z^m \longrightarrow \F_p^m$ denotes the reduction modulo~$p$) 
  \[ \pi_p\bigl(L_A(\Z)\bigr) = L_A(\F_p).
  \]
\end{lem}
\begin{proof}
  The structure theory of matrices over~$\Z$ shows that there exist
  invertible integral matrices~$S \in \GL_k(\Z)$ and $T \in \GL_m(\Z)$
  such that
  \[ \widetilde A := S \cdot A \cdot T
  \]
  is in Smith normal form; i.e., $\widetilde A$ is a ``diagonal'' matrix of the form
  \[ \widetilde A
  = \begin{pmatrix}
    a_1 
    \\
    & \ddots 
    \\
    & & a_r 
    \\
    & & & 0
    \\
    & & & & \ddots
    \\
    & & & & & 0
  \end{pmatrix}
  \in M_{k \times m}(\Z)
  \]
  with $r \in \{0, \dots, \min(m,k)\}$ and $a_1,\dots, a_r \in \Z \setminus \{0\}$
  satisfying $a_1 \mid a_2$, $a_2 \mid a_3$, \dots, $a_{r-1} \mid
  a_r$.

  If $R$ is a domain, then (we interpret integral matrices
  in the canonical way as matrices over~$R$ and use the fact
  that invertible integral matrices stay invertible over~$R$)
  \[ L_A(R) = L_{S^{-1} \cdot \widetilde A \cdot T^{-1}}(R)
            = T \cdot L_{\widetilde A}(R).
  \]
  In particular, we have~$L_A(R) = \{0\}$ if and only if
  \[ r = m \quad\text{and}\quad \fa{j \in \{1,\dots,r\}} a_j \neq 0 \text{ in~$R$}.
  \]

  We now consider the set~$Q \subset \N$ of prime divisors
  of~$a_1,\dots, a_r$.  Clearly, this set~$Q$ is finite. Then $P' := P
  \setminus Q$ is cofinite in~$P$ and for all~$p \in P'$ we
  have~$L_{\widetilde A}(\F_p) = \pi_p(L_{\widetilde A}(\Z))$. Hence,
  we obtain
  \begin{align*}
    L_A(\F_p)
    & = T \cdot L_{\widetilde A} (\F_p)
      = T \cdot \pi_p \bigl(L_{\widetilde A}(\Z)\bigr)
    = \pi_p \bigl( T \cdot L_{\widetilde A}(\Z)\bigr)
    \\
    & = \pi_p \bigl( L_A(\Z) \bigr)
  \end{align*}
  for all~$p \in P'$. 
  In particular, $L_A(\Z) \neq \{0\}$. 
\end{proof}

\subsection{Equal characteristic}

As a warm-up for the proof of Theorem~\ref{thm:fpq}, we prove the
following comparison result in equal characteristic:

\begin{thm}[equal characteristic]\label{thm:equalchar}
  Let $M$ be an oriented closed connected manifold and 
  let $F$ and $E$ be fields of the same characteristic. Then
   \[ \suv M F = \suv M E.
  \]
\end{thm}
\begin{proof}
  Let $n := \dim M$ and let $c = \sum_{j=1}^m a_j \cdot \sigma_j \in
  C_n(M;E)$ be an $E$-fundamental cycle of~$M$ in reduced form
  with~$|c| = m = \suv M E$. We consider the model~$Z$ of~$c$
  and the associated cycle matrix~$A \in M_{k \times m}(\Z)$ (for simplicity,
  we index the rows of~$A$ by natural numbers instead of by pairs). So far,
  $A$ only encodes the fact that $c$ is a cycle. In order to also
  encode the corresponding homology class, we proceed as follows:

  We pick a point~$x \in M$ and consider the local degrees~$d_j \in \Z$
  determined uniquely by the relation
  \[ H_n(\sigma_j;\Z)[\Delta^n,\partial \Delta^n] = d_j \cdot [M; x]_\Z 
     \in H_n\bigl(M,M\setminus \{x\};\Z\bigr);
  \]
  here, $[M;x]_\Z = H_n(i;\Z)[M]_\Z$ denotes the generator
  in~$H_n(M,M\setminus\{x\};\Z) \cong H_n(M;\Z) \cong \Z$
  corresponding to the fundamental class~$[M]_\Z$ under the
  inclusion~$i \colon (M,\emptyset) \longrightarrow (M,M\setminus
  \{x\})$.  Because $c$ is a fundamental cycle, we obtain
  \[ \sum_{j=1}^m d_j \cdot a_j = 1 \quad \text{in $E$}.
  \]
  Hence, the coefficients~$(a_1,\dots,a_m) \in E^m$ of~$c$ lie
  in the solution set~$L_{\overline A, \overline b}(E)$,
  where $\overline A \in M_{(k+1)\times m}(\Z)$ is the matrix obtained
  from~$A$ by adding the last row~$(d_1,\dots, d_m)$ and where $\overline b
  := (0,\dots,0,1) \in \Z^{k+1}$.
  
  In particular, $L_{\overline A, \overline b}(E) \neq
  \emptyset$. Because $F$ and $E$ have the same characteristic, we
  also have~$L_{\overline A, \overline b}(F) \neq \emptyset$
  (Lemma~\ref{lem:equalcharsolv}). Let $(a'_1,\dots, a'_m) \in L_{\overline A, \overline b}(F)$.
  We now consider the chain
  \[ c' := \sum_{j=1}^m a'_j \cdot \sigma_j \in C_n(M;F).
  \]
  By construction, the model of~$c'$ is contained in the model~$Z$ of~$c$. 
  The chain~$c'$ is a cycle because $(a'_1,\dots,a'_m) \in L_{\overline A,\overline
    b}(F) \subset L_A(F)$ and $A$ is the cycle matrix of the model~$Z$
  (Lemma~\ref{lem:cycleencoding}). Moreover, $c'$
  is an $F$-fundamental cycle of~$M$: The additional equation 
  \[ \sum_{j=1}^m d_j \cdot a'_j = 1 \quad \text{in $F$}
  \]
  ensures that 
  \begin{align*}
    H_n(i; F) [c']
  & = \sum_{j=1}^m a'_j \cdot H_n(\sigma_j;F)[\Delta^n,\partial \Delta^n]_F
  = \sum_{j=1}^m a'_j \cdot d_j \cdot [M;x]_F
  \\ & = [M;x]_F
  \quad \text{in~$H_n(M, M \setminus \{x\};F)$}, 
  \end{align*}
  and hence~$[c'] = [M]_F$.

  In particular,
  $\suv M F \leq |c'| \leq m = \suv M E.
  $ 
  By symmetry, we also obtain the reverse inequality~$\suv M E \leq \suv M F$.
\end{proof}

\subsection{Comparison with characteristic~$0$}

Using the tools and techniques from the previous sections, we will now
give a proof of Theorem~\ref{thm:fpq}. The situation in this proof is
slightly different from the equal characteristic case because we will
need to deal with cycles that have isomorphic models but consist of
different singular simplices. 

\begin{proof}[Proof of Theorem~\ref{thm:fpq}]
  We begin with the proof that $\svp M p \leq \suv M \Q$ holds for all
  but finitely many primes~$p$: Let $c \in C_n(M;\Q)$ be a
  $\Q$-fun\-da\-men\-tal cycle of~$M$ with~$|c| = \suv M \Q$.  Then there
  exists an~$m \in \N \setminus \{0\}$ such that $m \cdot c$ is an
  integral cycle, representing~$m \cdot [M]_\Z $ in~$H_n(M;\Z)$.  Let
  $P \subset \N$ be the set of primes that do \emph{not}
  divide~$m$. Then $P$ contains all but finitely many primes
  in~$\N$. Let $p \in P$ and let $c_p \in C_n(M;\F_p)$ be the mod~$p$
  reduction of~$c$. Then $c_p$ is a cycle, $m$ is a unit modulo~$p$, and
  \[ \frac 1m \cdot c_p \in C_n(M;\F_p)
  \]
  is an $\F_p$-fundamental cycle of~$M$. In particular,
  \[ \svp M p \leq \Bigl| \frac 1m \cdot c_p \Bigr| \leq |c| = \suv M \Q.
  \]

  We will now prove the converse estimate: 
  By Remark~\ref{rem:trivest}, we have
  \[ \bigl\{ \svp M p \bigm| p \in \N \text{ prime} \bigr\}
     \subset \bigl\{ 0,\dots, \isv M \bigr\};
  \]
  in particular, the set on the left hand side is finite. Let $V \in \{0, \dots, \isv M \}$
  be the smallest accumulation point of~$(\svp M p)_{p \in \N \text{ prime}}$ and let $P \subset \N$
  be the set of primes~$p \in \N$ with
  $V \leq \svp M p$. 
  Because $\{0,\dots,\isv M \}$ is finite, the set $P$ is
  cofinite in the set of primes. Therefore, it suffices to prove that
  for all but finitely many~$p \in P$ we have
  \[ \suv M \Q \leq V \leq \svp M p.
  \]
  Because there exist only finitely many isomorphism classes of
  $n$-dimensional model complexes with at most~$V$ simplices, there exists such
  an $n$-dimension\-al model complex~$Z$ and an infinite subset~$P'
  \subset P$ such that for every~$p \in P'$ there exists an
  $\F_p$-fundamental cycle of~$M$ whose model is isomorphic to~$Z$. We
  will now show that there also exists a $\Q$-fundamental cycle of~$M$
  whose model is isomorphic to~$Z$:

  Let $A \in M_{k \times m}(\Z)$ be the integral cycle matrix
  associated with~$Z$ (again, we simplify the index set).
  By Lemma~\ref{lem:cycleencoding}, we hence know
  that
  \[ \fa{p \in P'} L_A(\F_p) \neq \{0\}.
  \]
  Applying Lemma~\ref{lem:reduction} therefore shows that there exists
  a cofinite set~$P'' \subset P'$ with
  \[ \fa{p \in P''} \pi_p \bigl( L_A(\Z)\bigr) = L_A(\F_p).
  \]
  Let $p \in P''$ and let $c_p = \sum_{j=1}^V a_j \cdot \sigma_j\in C_n(M;\F_p)$ be an $\F_p$-fundamental
  cycle of~$M$ with model isomorphic to~$Z$.
  Because of~$\pi_p(L_A(\Z)) = L_A(\F_p)$ 
  and Lemma~\ref{lem:cycleencoding} there exists a \emph{cycle}~$c = \sum_{j=1}^V\widetilde a_j
  \cdot \sigma_j \in C_n(M;\Z)$
  whose reduction modulo~$p$ equals~$c_p$ and whose model is isomorphic to~$Z$. In particular,
  \[ [c] \neq 0 \in H_n(M;\Z) \subset H_n(M;\Q)
  \]
  (because its reduction~$[c_p]$ is non-zero in~$H_n(M;\F_p)$). Then a rational
  multiple of~$c$ is a $\Q$-fundamental cycle of~$M$ and we obtain
  $\suv M  \Q \leq |c| \leq V.$
\end{proof}

\begin{cor}
  Let $M$ be an oriented closed connected manifold and let $F$ be a field
  of characteristic~$0$. Then for all but finitely many primes~$p \in \N$
  we have
  \[ \svp M p = \suv M F.
  \]
\end{cor}
\begin{proof}
  We only need to combine Theorem~\ref{thm:fpq} with Theorem~\ref{thm:equalchar}.
\end{proof}

\section{Stabilisation along finite coverings}\label{sec:stable}

In this section, we discuss the stabilisation of weightless
simplicial volumes along towers of finite coverings. In particular,
we will prove Theorem~\ref{thm:selfmap} and Theorem~\ref{thm:dynview}.

\subsection{Stable weightless simplicial volumes}

The classical simplicial volume~$\sv \args$ is multiplicative under finite coverings;
however, integral simplicial volume~$\isv \args$, in general, is \emph{not}
multiplicative under finite coverings. Hence, it makes sense to study the
corresponding gradient invariant, the \emph{stable integral simplicial volume},
defined by
\[ \stisv M := \inf \Bigl\{ \frac{\isv N}d \Bigm|
\text{$d \in \N$, $N \rightarrow M$ is a $d$-sheeted covering}
\Bigr\}
\]
for oriented closed connected $n$-manifolds~$M$.

Analogously, we can also introduce the gradient invariants associated
with weigthless simplicial volumes:

\begin{defi}[stable weightless simplicial volume]
  Let $M$ be an oriented closed connected $n$-manifold and let $R$
  be a commutative ring with unit. Then the \emph{stable weightless $R$-simplicial volume}
  of~$M$ is defined by
  \[ \stsuv M R := \inf \Bigl\{ \frac{\suv N R}d \Bigm|
  \text{$d \in \N$, $N \rightarrow M$ is a $d$-sheeted covering}
  \Bigr\}.
  \]
  This generalises Definition~\ref{def:stsvp}.
\end{defi}

\begin{rem}[basic estimates]
  Let $M$ be an oriented closed connected $n$-manifold and let $R$ be a commutative ring
  with unit.
  \begin{enumerate}
  \item Then~$\sv M \leq \stisv M$ and~$\stsuv M R \leq \stisv M$.
  \item If $M$ admits a self-covering of non-trivial degree, then
    $\stsuv M R =0$. In particular, the stable weightless simplicial
    volumes of tori are zero.
  \item If $(\Gamma_j)_{j \in \N}$ is a descending chain of finite index subgroups of~$\pi_1(M)$
    and $(M_j)_{j \in \N}$ is the corresponding tower of finite covering manifolds of~$M$,
    then the sequence
    \[ \Bigl( \frac1{[\pi_1(M) : \Gamma_j ]} \cdot \suv {M_j} R\Bigr)_{j \in \N}
    \]
    is monotonically decreasing (full lifts of fundamental cycles give
    fundamental cycles of finite coverings), whence convergent. Therefore,
    \[ \lim_{j \rightarrow \infty} \frac{\suv {M_j} R}{[\pi_1(M) :\Gamma_j]}
     = \inf_{j \rightarrow \infty} \frac{\suv {M_j} R}{[\pi_1(M) :\Gamma_j]}.
    \]
  \item Every (generalised) triangulation of~$M$ gives rise to an $\F_2$-fundamen\-tal cycle.
    Hence, we obtain
    \[ \stsvp M 2 \leq \sigma_\infty (M),
    \]
    where $\sigma_\infty(M)$ denotes the stable $\Delta$-complexity of~$M$~\cite{milnorthurston}.
  \end{enumerate}
\end{rem}

\subsection{The Euler characteristic estimate}\label{subsec:chi}

We start with a few simple example estimates for stable weightless
simplicial volumes of even-dimensional hyperbolic manifolds.

\begin{exa}[surfaces]
  Let $\Sigma$ be an oriented closed connected (non-empty) surface
  of genus~$g \in \N_{\geq 1}$ and let $p \in \N$ be a prime. Then
  we know $2 \cdot g \leq \svp \Sigma p \leq 4 \cdot g - 2$ (Example~\ref{exa:surface}). 
  Taking the infimum over all finite coverings of~$\Sigma$ produces the estimates
  (and similarly for~$\stsuv M \Z$)
  \[ \bigl| \chi(\Sigma) \bigr|
  = 2 \cdot g - 2 \leq \stsvp \Sigma p \leq 4 \cdot g - 4
  = \sv \Sigma = 2 \cdot \bigl| \chi(\Sigma)\bigr|.
  \]
  However, the exact values of~$\stsvp \Sigma p$ or $\stsuv \Sigma \Z$ are
  not known if~$g \geq 2$.
\end{exa}

\begin{prop}[Euler characteristic estimate]\label{prop:chi}
  Let $M$ be an oriented closed connected $n$-manifold and let $R$
  be a principal ideal domain. Then
  \[ \bigl| \chi(M) \bigr| \leq (n+1) \cdot \suv M R. 
  \]
  In particular: For all prime numbers~$p \in \N$ we have
  $ \bigl| \chi(M) \bigr| \leq (n+1) \cdot \svp M p.
  $
\end{prop}
\begin{proof}
  Let $N \longrightarrow M$ be a finite covering of~$M$ with $d \in \N$ sheets.
  Then the Betti number estimate of Proposition~\ref{prop:PD} shows that
  \begin{align*}
    \bigl| \chi(M) \bigr|
    & = \frac 1d \cdot \bigl| \chi(N) \bigr|
    \leq \frac 1d \cdot \sum_{j=0}^n \rk_R H_j(N;R)
    \\
    & \leq (n+1) \cdot \frac 1d \cdot \suv N R.
  \end{align*}
  Taking the infimum over all finite coverings of~$M$ finishes
  the proof.
\end{proof}

\begin{cor}
  Let $n \in \N$ be even. Then there exists a constant~$C_n \in \N_{>0}$
  such that: For every oriented closed connected hyperbolic $n$-manifold
  and every principal ideal domain~$R$ we have
  \[ \suv M R \geq C_n \cdot \vol(M).
  \]
  In particular: For all prime numbers~$p \in \N$ we have~$\svp M p \geq C_n \cdot \vol(M)$.
\end{cor}
\begin{proof}
  Applying the generalised Gau\ss-Bonnet
  formula~\cite[Theorem~11.3.2]{ratcliffe} to the hyperbolic manifold~$M$
  results in 
  \[ \chi (M) = (-1)^{n/2} \cdot \frac 2 {\Omega_n} \cdot \vol(M),
  \]
  where $\Omega_n$ denotes the volume of the standard unit $n$-sphere. 

  In combination with Proposition~\ref{prop:chi}, we therefore obtain
  \[ \stsuv M R \geq \frac 1{n+1} \cdot \bigl|\chi(M)\bigr|
  = \frac 2 {(n+1) \cdot \Omega_n} \cdot \vol(M).
  \qedhere
  \]
\end{proof}

\subsection{Self-maps of non-trivial degree}

We will now prove Theorem~\ref{thm:selfmap} in the following, slightly
more general, form:

\begin{thm}\label{thm:selfmapgen}
  Let $M$ be an oriented closed connected aspherical manifold with residually
  finite fundamental group that admits a continuous self-map~$f \colon M \longrightarrow M$
  with~$\deg f \not\in \{-1,0,1\}$. If $R$ is a commutative ring with unit
  and $\deg f$ is a unit in~$R$, then
  \[ \stsuv M R = 0.
  \]
\end{thm}

We proceed in the same way as in the proof of the corresponding
statement for $L^2$-Betti numbers~\cite[Theorem~14.40]{lueck}. 
As a preparation, we recall the following well-known fact:

\begin{lem}[mapping degrees and index of subgroups]\label{lem:degindex}
  \hfil
  \begin{enumerate}
  \item
    Let $M$ and $N$ be oriented closed connected manifolds of the same
    dimension and let $f \colon M\longrightarrow N$ be a continuous map of non-zero
    degree. Then $\im \pi_1(f)$ has finite index in~$\pi_1(N)$.
  \item
    Let $M$ be an oriented closed connected aspherical manifold with residually
    finite fundamental group and let $f \colon M \longrightarrow M$
    be a continuous map with~$\deg f \not\in \{-1,0,1\}$. Then
    \[ 1 < \bigl[ \pi_1(M) : \im \pi_1(f) \bigr] < \infty.
    \]
  \end{enumerate}
\end{lem}
\begin{proof}
  The first part follows from covering theory (by considering
  the covering of~$N$ associated with~$\im \pi_1(f)$). 

  For the second part, we write $\Gamma := \pi_1(M)$. 
  By the first part, $\im \pi_1(f)$ has finite index in~$\Gamma$. 

  \emph{Assume} for a contradiction that $[\Gamma : \im \pi_1(f)] = 1$,
  i.e., that $\pi_1(f)$ is an epimorphism. As residually finite group,
  the group~$\Gamma$ is Hopfian. Hence, the epimorphism~$\pi_1(f)
  \colon \Gamma \longrightarrow \Gamma$ is an isomorphism. Because $M$
  is aspherical, this implies that $f$ is a homotopy
  equivalence.  In particular, $\deg f \in \{-1,1\}$, which is a
  contradiction. Therefore, $[\Gamma : \im \pi_1(f)] > 1$.
\end{proof}

\begin{proof}[Proof of Theorem~\ref{thm:selfmapgen}]
  Let $\Gamma := \pi_1(M)$. 
  For~$k \in \N$ let~$f_k := f^{\circ k} \colon M \longrightarrow M$
  be the $k$-fold composition of~$f$ and let
  $\Gamma_k := \im \pi_1(f_k) \subset \Gamma$ be the corresponding
  subgroup. 
  For~$k \in \N$, 
  let $p_k \colon M_k \longrightarrow M$ be the covering associated
  with the subgroup~$\Gamma_k \subset \Gamma$; thus, there
  is a $p_k$-lift~$\overline f_k \colon M \longrightarrow M_k$ of~$f_k$.
  By construction,
  \begin{align*}
        \deg \overline f_k \cdot \deg p_k 
    & 
      = \deg (p_k \circ \overline f_k)
      = \deg f_k
    \\
    & = (\deg f)^k,
  \end{align*}
  and so $\deg \overline f_k$ is a unit in~$R$.
  In view of Proposition~\ref{prop:degreemon},
  we obtain
  \[ \suv {M_k} R \leq \suv M R;  
  \]
  thus, 
  \[ \stsuv M R
  \leq \inf_{k \in \N} \frac{\suv {M_k} R}{\mathopen{|}\deg p_k|}
  \leq \inf_{k \in \N} \frac{\suv{M} R}{[\Gamma : \Gamma_k]}.
  \]

  Therefore, it suffices to show that
  the sequence~$([\Gamma :\Gamma_k])_{k \in \N}$ is unbounded: 
  To this end, for~$k \in \N$, we consider the self-map
  \[ g_k := \overline f_k \circ p_k \colon M_k \longrightarrow M_k
  \]
  of~$M_k$. Because $\Gamma_k$ has finite index in~$\Gamma$, the
  covering manifold~$M_k$ is compact; morevoer, $M_k$ is aspherical
  and oriented and $\pi_1(M_k) \cong \Gamma_k \subset \Gamma$ is residually
  finite. Using the fact that
  \[ \deg g_k = \deg \overline f_k \cdot \deg p_k = (\deg f)^k \not \in \{-1,0,1\},
  \]
  we obtain from Lemma~\ref{lem:degindex} that
  \begin{align*}
    1 & < \bigl[ \pi_1(M_k) : \im \pi_1(g_k) \bigr]
    = \bigl[ \Gamma_k : \pi_1(f_k)(\Gamma_k) \bigr]
    = [\Gamma_k : \Gamma_{2 \cdot k} ].
  \end{align*}
  In particular, the sequence~$([\Gamma : \Gamma_{2^k}])_{k \in \N}$
  is unbounded. 
\end{proof}

\begin{proof}[Proof of Theorem~\ref{thm:selfmap}]
  If $p \nmid \deg f$, then $\deg f$ is a unit in~$\F_p$ and so Theorem~\ref{thm:selfmapgen}
  can be applied.
\end{proof}

\subsection{The dynamical view}\label{subsec:dyndef}

In order to formulate and prove Theorem~\ref{thm:dynview}, we first introduce the type of
dynamical systems that we are interested in:

\begin{defi}
  Let $\Gamma$ be a group. A \emph{standard $\Gamma$-space}
  is a standard Borel probability space~$(X,\mu)$ together with
  a $\mu$-preserving $\Gamma$-action.

  If $\alpha = \Gamma \actson (X,\mu)$ is a standard $\Gamma$-space
  and $R$ is a commutative ring with unit, then we write
  \[ L^\infty(\alpha,R) := L^\infty(X,\Z) \otimes_\Z R.
  \]
  This module is equipped with the $\Z\Gamma$-right module structure
  induced by the $\Gamma$-action~$\alpha$ on~$X$.  More concretely,
  $L^\infty(\alpha,R)$ can be viewed as the $\Z\Gamma$-module of functions~$X
  \longrightarrow R$ with finite image and measurable pre-images (up
  to equality $\mu$-almost everywhere).

  If $f \colon X \longrightarrow R$ is an element of~$L^\infty(\alpha,R)$, then
  we write
  \[ \supp f := f^{-1} ( R \setminus \{0\}) \subset X.
  \]
\end{defi}

\begin{exa}[profinite completion]
  If $\Gamma$ is a finitely generated residually finite group, then
  the profinite completion~$\pfc \Gamma$ of~$\Gamma$ is a standard
  Borel space. The canonical translation action of~$\Gamma$
  on~$\pfc \Gamma$ is measure preserving with respect to the inverse
  limit probability measure of the normalised counting measures on the
  finite quotients of~$\Gamma$. Moreover, this action
  is essentially free. For simplicity, we also denote the corresponding
  standard $\Gamma$-space by~$\pfc \Gamma$.
\end{exa}

In analogy with parametrised/integral foliated simplicial
volume~\cite{mschmidt,loehpagliantini}, we consider weightless
parametrised simplicial volumes (by ignoring the magnitude
of the coefficients):

\begin{defi}[weightless parametrised simplicial volume]
  Let $M$ be an oriented closed connected $n$-manifold, let $\alpha :=
  \pi_1(M) \actson (X,\mu)$ be a standard $\pi_1(M)$-space, and let
  $R$ be a commutative ring with unit.  A cycle~$c \in
  C_n(M;L^\infty(\alpha,R)) = L^\infty(X,R) \otimes_{\Z \pi_1(M)}
  C_n(\widetilde M;\Z)$ is an \emph{$(\alpha;R)$-fundamental cycle
    of~$M$} if~$c$ is homologous (in~$C_*(M; L^\infty(\alpha,R))$) to
  an integral fundamental cycle of~$M$ (via the canonical map~$C_*(M;\Z)
  \longrightarrow C_*(M;L^\infty(\alpha,R))$ given by the constant functions).   
  The \emph{weightless
    parametrised $R$-simplicial volume} of~$M$ is defined as
  \begin{align*}
    \suv M {\alpha;R}
  := \inf \biggl\{ \sum_{j=1}^m \mu(\supp f_j)
    \biggm| \;
    & \sum_{j=1}^m f_j \otimes \sigma_j \in C_n\bigl(M; L^\infty(\alpha,R)\bigr)
    \\
    & \text{is an $(\alpha;R)$-fundamental cycle of~$M$}
  \biggr\}. 
  \end{align*}
\end{defi}

Clearly, in the situation of the previous definition, we have 
\[ 0 \leq \suv M {\alpha;R} \leq \suv M {\alpha;\Z} \leq \ifsv M ^\alpha,
\]
where $\ifsv M ^\alpha$ denotes the ordinary parametrised simplicial volume
with parameter space~$\alpha$.

Finally, we prove Theorem~\ref{thm:dynview}; again, we 
establish a slightly more general version:

\begin{thm}\label{thm:dynviewgen}
  Let $M$ be an oriented closed connected manifold with residually finite fundamental
  group and let $R$ be a commutative ring with unit. Then
  \[ \stsuv M R = \suv M {\pfc{\pi_1(M)};R}.
  \]
\end{thm}

The proof is a straightforward adaption of the proof of the
corresponding statement for stable integral simplicial
volume~\cite[Theorem~6.6, Remark~6.7]{loehpagliantini}. We first
set up some notation for the proof:
\begin{itemize}
\item
  We abbreviate
  $\Gamma := \pi_1(M)$ and $\alpha := \Gamma \actson \pfc \Gamma$.
\item
  If $\Lambda \subset \Gamma$ is a finite index subgroup,  we write
  $M_\Lambda \longrightarrow M$ for the associated covering. 
\item
  We write $S$ for the set of all finite index subgroups and $F(S)$
  for the set of all finite subsets of~$S$.
\item
  If $F \in F(S)$, then we write $X_F := \varprojlim_{\Lambda \in F} \Gamma/\Lambda$
  and $\alpha_F := \Gamma \actson X_F$ for the associated parameter space. Moreover,
  we denote the canonical $\Gamma$-map~$\pfc \Gamma \longrightarrow X_F$ by~$\pi_F$.
\end{itemize}
In the proof of Theorem~\ref{thm:dynviewgen}, we will use the
following ingredients:
\begin{enumerate}
\item\label{ing:Lambda} If $\Lambda$ is a finite index subgroup of~$\Gamma$, then
  \[ \suv M {\alpha_{\{\Lambda\}};R}
     = \frac 1 {[\Gamma :\Lambda]} \cdot \suv {M_\Lambda} R.
  \]
  [The proof of the corresponding statement for parametrised integral
    simplicial volume~\cite[Proposition~4.26, Corollary~4.27]{loehpagliantini}
    carries over to the weightless setting, because the division by the
    index happens inside the probability space.]
\item\label{ing:F} If $F \in F(S)$, then $\alpha_F$ is a finite
  $\Gamma$-probability space and thus a finite convex combination of
  coset spaces as in~(\ref{ing:Lambda}). Therefore,
  \[ \stsuv M R \leq \suv M {\alpha_F;R}.
  \]
  [Again, the arguments of the classical
    case~\cite[Proposition~4.15]{loehpagliantini} also work in the
    weightless setting.]
\item\label{ing:L} Let $L \subset L^\infty(\alpha,R)$ be a $\Z\Gamma$-submodule
  that is dense in the following sense:
  \[ \fa{f \in L^\infty(\alpha,R)} \fa{\varepsilon \in \R_{>0}}
     \exi{g \in L} \mu\bigl( \supp (f-g)\bigr) < \varepsilon.
  \]
  Then the induced homomorphism~$H_n(M;L) \longrightarrow
  H_n(M;L^\infty(\alpha,R))$ is norm-preserving with respect
  to the weightless norms.

  \noindent
  [The standard
    proof~\cite[Lemma~2.9]{mschmidt}\cite[Proposition~1.7]{loehphd} of
    statements of this type -- by approximating boundaries -- also
    works in the weightless setting.]
\end{enumerate}

\begin{proof}[Proof of Theorem~\ref{thm:dynviewgen}]
  We first prove that~$\suv M {\alpha,R} \leq \stsuv M R$:
  If $\Lambda \in S$, then the canonical projection~$\alpha \longrightarrow \alpha_{\{\Lambda\}}$
  and~(\ref{ing:Lambda}) show that
  \[ \suv M {\alpha;R} \leq \suv M {\alpha_{\{\Lambda\}};R}
      = \frac1{[\Gamma:\Lambda]} \cdot \suv {M_\Lambda} R.
  \]
  Every (connected) finite covering of~$M$ corresponds to a finite
  index subgroup of~$\pi_1(M)$. Therefore, taking the infimum over all~$\Lambda$
  in~$S$ implies 
  \[ \suv M {\alpha;R} \leq \inf_{\Lambda \in S} \frac1{[\Gamma:\Lambda]} \cdot \suv {M_\Lambda} R
     = \stsuv M R.
  \]

  Conversely, we now show that $\stsuv M R \leq \suv M{\alpha,R}$: 
  To this end we consider the $\Z\Gamma$-submodule
  \[ L := \bigcup_{F \in F(S)} \bigl\{ f \circ \pi_F \bigm| f \in L^\infty(\alpha_F,R) \bigr\}
  \]
  of~$L^\infty(\alpha,R)$. The submodule~$L$ is dense
  in~$L^\infty(\alpha,R)$ in the sense of~(\ref{ing:L}): Let $\sigma$ be the
  $\sigma$-algebra of~$\pfc \Gamma$. The set
  \[ \sigma' := \bigl\{ A \in \sigma
  \bigm| \fa{\varepsilon \in \R_{>0}}
         \exi{B \in \sigma} \chi_B \in L \land \mu(A \triangle B) < \varepsilon \bigr\} 
  \]
  is a $\sigma$-algebra. Moreover, $\sigma'$ is contained in~$\sigma$
  and $\sigma'$ contains for every~$F \in F(S)$ the $\pi_F$-preimage
  of the $\sigma$-algebra of~$X_F$. Hence, $\sigma' = \sigma$ and so
  $L$ is dense in~$L^\infty(\alpha,R)$.

  Combining (\ref{ing:F}), (\ref{ing:L}), and the construction of~$L$, we obtain
  \[ \stsuv M R \leq \inf_{F \in F(S)} \suv M {\alpha_F;R} = \suv M {\alpha;R}.
  \qedhere
  \]
\end{proof}

\begin{exa}[small amenable covers]\label{exa:smallamenable}
  Let $M$ be an oriented closed connected aspherical (triangulable)
  $n$-manifold with residually finite fundamental group that admits an
  open cover by amenable sets such that each point of~$M$ is contained
  in at most~$n$ of these sets. Then
  \[ \stsuv M R = 0
  \]
  for all commutative rings~$R$ with unit (in particular, $\stsvp M p = 0$ for
  all prime numbers~$p$). Before giving a proof of this statement, we recall
  that a subset~$U \subset M$ of~$M$ is \emph{amenable} if for every~$x \in U$
  the image of the homomorphism~$\pi_1(U,x) \longrightarrow \pi_1(M,x)$ induced
  by the inclusion is amenable.

  Indeed, by work of Sauer~\cite[proof of Theorem~B, Section~5.3]{sauer} we have
  \[ \suv M {\pfc{\pi_1(M)}; \Z} = 0
  \]
  (Sauer's notion of mass of the fundamental class of~$M$ with respect
  to the action of~$\pi_1(M)$ on~$\pfc{\pi_1(M)}$ coincides with~$\suv
  M {\pfc{\pi_1(M)};\Z}$). Therefore, the canonical ring homomorphism~$\Z \longrightarrow R$
  shows that
  \[ 0 \leq \suv M {\pfc{\pi_1(M)};R} \leq \suv M {\pfc{\pi_1(M)};\Z} = 0.
  \]
  Applying Theorem~\ref{thm:dynviewgen}, we hence obtain
  \[ \stsuv M R = \suv M {\pfc{\pi_1(M)};R} = 0.
  \]
\end{exa}
  
\begin{ack}
  I am grateful to the anonymous referee for carefully reading
  the manuscript.
\end{ack}


\medskip
\vfill

\noindent
\emph{Clara L\"oh}\\[.5em]
  {\small
  \begin{tabular}{@{\qquad}l}
    Fakult\"at f\"ur Mathematik,
    Universit\"at Regensburg,
    93040 Regensburg\\
    \textsf{clara.loeh@mathematik.uni-r.de},\\
    \textsf{http://www.mathematik.uni-r.de/loeh}
  \end{tabular}}


\begin{thebibliography}{100}

  \bibitem{benedettipetronio} 
    R.~Benedetti, C.~Petronio.
    \emph{Lectures on Hyperbolic Geometry}, Universitext,
    Springer, 1992.
    
  \bibitem{fauser}
    D.~Fauser.
    Integral foliated simplicial volume and $S^1$-actions,
    preprint, available at \textsf{arXiv:1704.08538 [math.GT]}, 
    2017.

  \bibitem{ffl}
    D.~Fauser, S.~Friedl, C.~L\"oh.
    Integral approximation of simplicial volume of graph manifolds,
    preprint, available at \textsf{arXiv:1807.10522 [math.GT]},
    2018.
    To appear in \emph{Bull.\ London Math.\ Soc.}.
    
  \bibitem{FLPS}
    R.~Frigerio, C.~L\"oh, C.~Pagliantini, R.~Sauer.
    Integral foliated simplicial volume of aspherical manifolds,
    \emph{Israel J.\ Math.}, 216(2), 707--751, 2016. 
     
  \bibitem{vbc}
    M.~Gromov.
    Volume and bounded cohomology,
    \emph{Inst. Hautes \'Etudes Sci. Publ. Math.}, 56, 5--99, 1983.

  \bibitem{gromovggt}
    M.~Gromov.
    \emph{Asymptotic invariants of infinite groups}. Geometric group theory, Vol.~2
    (Sussex 1991).
    London Math. Soc. Lectures Notes Ser., 182, Cambridge 
    University Press, Cambridge, 1--295, 1993.

  \bibitem{loehphd} C.~L\"oh. \emph{$\ell^1$-Homology and Simplicial
      Volume}, PhD~thesis, Westf\"alische Wilhelms-Universit\"at M\"unster, 2007.\\ 
      \textsf{http://nbn-resolving.de/urn:nbn:de:hbz:6-37549578216}
    
  \bibitem{mapsimvol}
    C.~L\"oh.
    Simplicial Volume,
    \emph{Bull.\ Man.\ Atl.}, 7--18, 2011.
    
  \bibitem{loehsmallisv}
    C.~L\"oh.
    Odd manifolds of small integral simplicial volume,
    \emph{Arkiv f\"or Matematik}, 56(2), 351--375, 2018.
    
  \bibitem{loehpagliantini}
    C.~L\"oh, C.~Pagliantini.
    Integral foliated simplicial volume of hyperbolic $3$-manifolds,
    \emph{Groups Geom.\ Dyn.}, 10(3), 825--865, 2016. 

  \bibitem{lueck}
    W.~L\"uck.
    \emph{$L^2$-Invariants: Theory and Applications to Geometry and
      $K$-Theory}. \emph{Ergebnisse der Mathematik und ihrer Grenzgebiete},
    3.~Folge,~44. Springer, 2002.

  \bibitem{milnorthurston}
    J.~Milnor, W.~Thurston.
    Characteristic numbers of $3$-manifolds,  
    \emph{Enseignement Math.\ (2)}, 23(3--4), 249--254, 1977.
    
  \bibitem{ratcliffe}
    J.G.~Ratcliffe.
    \emph{Foundations of Hyperbolic Manifolds}.
    Graduate Texts in Mathematics, 149, Springer, 1994.

  \bibitem{sauer}
    R.~Sauer.
    Amenable covers, volume and $L^2$-Betti numbers of aspherical manifolds, 
    \emph{J.~Reine Angew.\ Math (Crelle's Journal)}, 636, 47--92, 2009.
    
  \bibitem{mschmidt} 
    M.~Schmidt. 
    \emph{$L^2$-Betti Numbers of
    $\mathcal{R}$-spaces and the Integral Foliated Simplicial
    Volume}. PhD~thesis, Westf\"alische Wilhelms-Universit\"at
    M\"unster, 2005.\\ 
    \textsf{http://nbn-resolving.de/urn:nbn:de:hbz:6-05699458563}

\end{thebibliography}
\end{document}